\documentclass[11pt,twoside,a4paper]{amsart}

\usepackage{amsfonts,amsmath,amsthm}
\usepackage{hyperref}
\usepackage{anysize} 

\marginsize{3cm}{2.5cm}{3cm}{3cm} 

\newtheorem{thm}{Theorem}[section]
\newtheorem{lem}[thm]{Lemma}
\newtheorem{pro}[thm]{Proposition}
\newtheorem{cor}[thm]{Corollary}

\theoremstyle{definition}

\numberwithin{equation}{section}

\theoremstyle{remark}
\newtheorem{remark}{Remark}[section]

\theoremstyle{definition}

\newcommand{\C}{\mathbb{C}}
\newcommand{\I}{\mathbb{I}}

\newcommand{\N}{\mathbb{N}}
\newcommand{\Rd}{\mathbb{R}^{3}}

\newcommand{\V}{\mathbb{V}}

\numberwithin{equation}{section}

\begin{document}

\title[Self-adjoint extensions of Dirac operators]{Self-adjoint extensions of Dirac operators with Coulomb type singularity}

\author{Naiara Arrizabalaga, Javier Duoandikoetxea and Luis Vega}

\address{N. Arrizabalaga, Javier Duoandikoetxea and Luis Vega: Universidad del Pa\'is Vasco (UPV/EHU), Departamento de Matem\'aticas, Apartado 644, 48080, Spain}
\email{naiara.arrizabalaga@ehu.es, javier.duoandikoetxea@ehu.es, luis.vega@ehu.es}
\thanks{The authors are supported in part by the grant MTM2011-24054 of the Ministerio de Econom\'{\i}a y Competitividad (Spain).}

\begin{abstract} In this work we construct self-adjoint extensions of the Dirac operator associated to Hermitian matrix potentials with Coulomb decay  and prove that the domain is maximal.  The result is obtained by means of a Hardy-Dirac type inequality. In particular, we can work with some electromagnetic potentials such that both, the electric potential and the magnetic one, have Coulomb type singularity. 
\end{abstract}


\maketitle

\section{Introduction}\label{Intro}

This work is devoted to the construction of self-adjoint extensions for the Dirac operator with Coulomb type singularities. We will denote these potentials by $\V$, so the Dirac operator associated to $\V$ takes the form
$$H=-i\alpha\cdot\nabla+m\beta -\V,$$
for $m\geq 0$ and  $\alpha=(\alpha_1,\alpha_2,\alpha_3)$, where $\alpha_k,\beta\in\mathcal M_{4\times4}(\C)$, $k=1,2,3$, are the Dirac matrices  \begin{equation*}
\alpha_k=
  \left(
  \begin{array}{cc}
    0 & \sigma_k
    \\
    \sigma_k & 0
  \end{array}\right),
  \qquad
   \beta
  =
  \left(
  \begin{array}{cc}
    \I_2 & 0
    \\
    0 & -\I_2
  \end{array}\right),
\end{equation*}
defined in terms of the {\it Pauli matrices} $\sigma_k\in\mathcal M_{2\times2}(\C)$, given by
\begin{equation*}
\I_2
=
  \left(
  \begin{array}{cc}
    1 & 0
    \\
    0 & 1
  \end{array}\right),
  \quad
  \sigma_1
  =
  \left(
  \begin{array}{cc}
    0 & 1
    \\
    1 & 0
  \end{array}\right),
  \quad
  \sigma_2
  =
  \left(
  \begin{array}{cc}
    0 & -i
    \\
    i & 0
  \end{array}\right),
  \quad
  \sigma_3
  =
  \left(
  \begin{array}{cc}
    1 & 0
    \\
    0 & -1
  \end{array}\right),
\end{equation*}
and $m$ is a non-negative real constant describing the mass of the particle. 

In this paper we deal with Hermitian matrix potentials $\V$ such that 
$$\sup_{x\in\Rd\backslash \{0\}}|x||\V(x)|< 1,$$
where $|\V(x)|=\sup_{\|b\|=1}\langle\V b,\V b\rangle^{1/2}, \, b\in\C^4$. We construct self-adjoint extensions for the Dirac operator $H$ and give the explicit expression of the domain of the operator  $\mathcal{D}(H)$.
 
A particular example of potentials considered here are the electromagnetic potentials 
$$\V=\left( \begin{array}{rr} \frac{\nu}{|x|} &  \sigma\cdot A  \\  \sigma\cdot A  & \frac{\nu}{|x|}  \end{array} \right),$$
such that $\sup_{x\in\Rd\backslash\{0\}}|\nu|+|x||A(x)|<1$ where $\frac{\nu}{|x|}$ is the electrostatic potential and the magnetic one is defined as
$$A=A(x)=(A^1(x),A^2(x),A^3(x)) :\Rd\rightarrow \Rd.$$ 
For this concrete potential, the operator $H$ can be written as  
$$H=-i\alpha\cdot\nabla+m\beta -\V=-i\alpha\cdot\nabla_A+m\beta -\frac{\nu}{|x|}\I_4,$$
where $$\nabla_A=\nabla-iA.$$
Observe that the magnetic potential is introduced in the operator replacing the standard gradient by $\nabla_A$. 

Note that in the particular case in which $A=0$ we construct a self-adjoint extension of the Dirac operator with the Coulomb potential for $\nu<1$. 
One of the first results in this sense is due to Kato, who in his book \cite{Ka} proved that for $V$ a multiplication operator with an Hermitian $4\times 4$ matrix such that each component $V_{ik}$ is a function satisfying the estimate
$$|V_{ik}(x)|\leq a \frac{1}{2|x|}+b\quad \text{for all}\quad x\in\Rd\backslash \{0\}, \; i,k=1, 2, 3, 4,$$
for some constants $b>0$ and $a<1$, then $H=H_0+V$ is essentially self-adjoint  on $\C_c^{\infty}(\Rd\backslash \{0\},\C^4)$ and self-adjoint on $\mathcal{D}(H_0)=H^1(\Rd,\C^4)$. The proof of this result can be found in \cite{T} and it is based on the Kato-Rellich theorem. In 1971, Weidmann proved that the Dirac operator with the Coulomb potential $\frac{\nu}{|x|}$ defined on $\C_c^{\infty}(\Rd\backslash \{0\},\C^4)$ is essentially self-adjoint if and only if $\nu\leq \frac{\sqrt{3}}{2}$, see \cite{W}. One year later, Schmincke in \cite{S1} proved that 
$$|V(x)|\leq \frac{\nu}{|x|}, \quad \nu\in \left(0,\frac{\sqrt{3}}{2}\right),$$
implies that the Dirac operator $H_0+V$ is essentially self-adjoint.
Moreover, different authors like Schmincke, Weidmann, W\"ust, Nenciu and Klaus  among others construct distinguished self-adjoint extensions of this operator  for the Coulomb potential $\frac{\nu}{|x|}$ where $\nu<1$ by using different methods. Esteban and Loss define an extension for $\nu\leq 1$ via Hardy-Dirac inequalities. See for instance  \cite{S2}, \cite{S1}, \cite{W1}, \cite{W2}, \cite{W3}, \cite{N}, \cite{KW} and \cite{EL}. This problem is also treated in \cite{NA}, in fact, the Coulomb potential is a particular case of the potentials studied there. Although it is not written in  \cite{NA}, doing some small modifications it is possible to construct a self-adjoint extension for $\nu\leq 1$. 

Among the previously mentioned works we are mainly interested in two of them. In the early seventies, W\"ust in \cite{W1}, \cite{W2}, \cite{W3} and Nenciu in \cite{N} constructed distinguished self-adjoint extensions of the Dirac-Coulomb operator for $\nu<1$. The extension of W\"ust is characterized by the fact that the domain of the extension is included in $\mathcal{D}(r^{-1/2})$ which is defined as the space of functions $\psi$ such that 
$$\int_{\Rd}|\psi|^2\frac{dx}{|x|}<+\infty.$$
Meanwhile, Nenciu characterizes his extension by the fact that the domain is included in the space $H^{1/2}(\Rd,\C^4)$. A couple of years later Klaus and W\"ust showed in \cite{KW} that the extensions considered by W\"ust and Nenciu are the same. 

Self-adjoint extensions of the Dirac operator with magnetic potentials that have constant magnetic field are studied in  \cite{DEL}  and \cite{T}. More general matrix-valued potentials are considered in \cite{N} and in a series of papers by Arai and Yamada \cite{A1}, \cite{A2}, \cite{A3}. They prove essential self-adjointness, existence of distinguished self-adjoint extensions and invariance of the essential spectrum for a class of matrix-valued potentials. More concretely, in \cite{A2} the author defines a self-adjoint extension of the Dirac operator for a matrix-valued potential with a Coulomb type singularity, and moreover, the domain is contained in $H^{1/2}(\Rd,\C^4)$ and in $\mathcal{D}(r^{-1/2})$.

Next we state the two main results of this paper. 
\begin{thm}\label{teorema}
Let $\V$ an Hermitian matrix potential such that 
\begin{equation}\label{baldintza}
\sup_{x\in\Rd\backslash \{0\}}|x||\V(x)|< 1
\end{equation}
and  let $f\in L^2(\Rd, \C^4)$.  Then there exists a unique $\psi\in L^2(\Rd, \C^4)$ that satisfies 
\begin{equation}\label{theorem}
\int_{\Rd}\psi\cdot \overline{(H+ i) \varphi}=\int_{\Rd} f \cdot \overline{\varphi} 
\end{equation}
for all $\varphi\in L^2(\Rd,\C^4)$ such that $(H+i)\varphi \in L^2(\Rd,\C^4)$.
Analogously, there exists a unique $\psi\in L^2(\Rd, \C^4)$ that satisfies 
\begin{equation}\label{theorem2}
\int_{\Rd}\psi\cdot \overline{(H-i) \varphi}=\int_{\Rd} f \cdot \overline{\varphi}
\end{equation}
for all $\varphi\in L^2(\Rd,\C^4)$ such that $(H+i)\varphi \in L^2(\Rd,\C^4)$. Moreover, for a constant $c$ both $\psi$'s satisfy:
\begin{itemize}
\item[(i)] $\|\psi\|_{L^2}\leq  \|f\|_{L^2}$.
\item[(ii)] $\displaystyle{\int_{\Rd}\frac{|\psi|^2}{|x|}\leq c \int_{\Rd}|f|^2}$.
\item[(iii)] $\displaystyle{\int_{|x|\leq1}|x||\alpha\cdot\nabla\psi|^2\leq c \int_{\Rd}|f|^2}$.
\item[(iv)] $\|\psi\|_{H^{1/2}}\leq c \|f\|_{L^2}$.
\item[(v)] Let $f_1, f_2\in L^2(\Rd, \C^4)$ and the corresponding $\psi_1,\psi_2\in H^{1/2}(\Rd, \C^4)$. Then
$$\int_{\Rd}(H\pm i)\psi_1\cdot\overline{\psi_2}=\int_{\Rd}\psi_1\cdot\overline{(H\mp i)\psi_2}.$$
\end{itemize}
\end{thm}

Note that in the  previous theorem we  have written  $(H \pm i )\psi$ instead of $(H \pm i \I_4)\psi$ to shorten notation.

\begin{thm}\label{3thm}
Let $\V$ an Hermitian matrix potential that satisfies (\ref{baldintza}). Then the Dirac operator $H=-i\alpha\cdot\nabla+m\beta-\V$ with domain $\mathcal{D}(H)=\{\psi\in L^2(\Rd,\C^4): H\psi\in L^2(\Rd,\C^4)\}$ is self-adjoint.
Moreover,
\begin{equation}\label{3cond}
\mathcal{D}(H)\subset H^{1/2}(\Rd,\C^4)\cap \mathcal{D}(r^{-1/2}).
\end{equation}
\end{thm}

Theorem \ref{3thm} will be a consequence of Theorem \ref{teorema} by using the basic criterion for self-adjointness.  In a first step in the proof of Theorem \ref{teorema} we prove that the operator $H$ defined on the domain
$$\widetilde{\mathcal{D}}=\{\psi\in L^2(1+|x|): H\psi\in L^2(1+|x|)\}$$
is essentially self-adjoint and then we construct the extension by a density argument. 
Observe that  from (\ref{3cond}) our extension coincides with the one of Arai and  when $\V=\frac{\nu}{|x|}, \nu<1$ with the ones of W\"ust and Nenciu because it is characterized by the fact that the domain is contained in both spaces, $\mathcal{D}(r^{-1/2})$ and $H^{1/2}(\Rd,\C^4)$. One bonus of our approach is that we prove that the extension has maximal domain. The key ingredient for doing this is the inequality proved in the next theorem that we consider interesting by itself. 

\begin{thm}\label{final_ineq}
Let $0\leq m <+\infty$ then
 \begin{equation}\label{final}
 \int_{\Rd}\frac{1}{|x|}|\psi|^2 \leq \int_{\Rd}|x||(i \alpha\cdot\nabla-m\beta\pm \epsilon i)\psi|^2, \quad \epsilon>0.
 \end{equation}
The inequality is sharp in the sense that the constant on the right hand side can not be improved.
\end{thm}

\begin{remark}\label{ohar2}
As it will be seen in the proof a relevant minimizing sequence of  (\ref{final}) can be obtained by using the spinors $\psi_0^{\epsilon,m}$ where
$$\psi_0^{\epsilon,m}=\left( \begin{array}{rr} \phi_0^{\epsilon,m}  \\ \chi_0^{\epsilon,m}  \end{array} \right), \; \phi_0^{\epsilon,m}=C r^{-1}e^{-\sqrt{\epsilon^2+m^2} r} \ (C\in\C^2) \; \text{and} \; \chi_0^{\epsilon,m}=\frac{\epsilon+im}{\sqrt{\epsilon^2+m^2}}\left(\sigma\cdot\frac{x}{|x|}\right)\phi_0^{\epsilon,m}.$$

Note that $\psi_0^{\epsilon,m}$ are solutions of 
\begin{equation}\label{eigen}
(-i\alpha\cdot\nabla+m\beta-\V- i \epsilon\I_4)\psi_0^{\epsilon,m}=0
\end{equation}
for the Hermitian potentials
$$\V=\frac{1}{|x|}\left( \begin{array}{rr} c\I_2 & \overline{b}\sigma\cdot\frac{x}{|x|}  \\ b\sigma\cdot\frac{x}{|x|} & c\I_2  \end{array} \right)$$
with $c$ real   and 
\[
b= \frac {c(\epsilon+im)}{\sqrt{\epsilon^2+m^2}}-i.
\]
Those potentials satisfy $\sup_{x\in \Rd}|x||\V(x)|= 1$ only in the cases $c=1, b=0$ and $c=0, b=-i$, and in both cases  $|x||\V(x)|\equiv 1$.

The case $c=1, b=0$ corresponds to $\epsilon=0$. Therefore, for $\V=\frac{1}{|x|}\I_4$ the functions $\psi_0^{0,m}$ such that $\phi_0^{0,m}=C\frac{1}{r}e^{-mr}$ and $\chi_0^{0,m}=i\sigma\cdot\frac{x}{|x|} \phi_0$ are eigenfunctions of eigenvalue $0$ of $H$. 

In the case $c=0$, $b=-i$, (\ref{eigen}) holds for $\psi_0^{\epsilon,m}$ with positive $\epsilon$. We conclude that $\text{Ker}(H^*+i\epsilon)\neq\{0\}$. Furthermore, it is easy to see that in this case $H$ is symmetric in $\widetilde{\mathcal{D}}\cap \mathcal{D}(r^{-1/2})$. Therefore, $H$ is not essentially self-adjoint and in this sense condition ($\ref{baldintza}$) is sharp. 

\end{remark}

The paper is structured as follows. In Section \ref{sec:ineq} we prove the inequalities we will need for the proofs of Theorems \ref{final_ineq} and \ref{teorema}.  Sections \ref{sec:pde} and \ref{sec:s_a} are devoted to the proof of Theorems \ref{teorema} and \ref{3thm}, respectively. To finish, in Section \ref{sec:further} we follow the arguments of \cite{DDEV} and give an alternative proof of Theorem \ref{final_ineq}.

\section{Some Hardy-Dirac inequalities}\label{sec:ineq}

We first state and prove several lemmas that we will need for the proof of Theorem \ref{final_ineq}.

\begin{lem}\label{Lemma2}
Let $\phi:\Rd\to \C^2$, $r=|x|$ and $\partial_r=\frac{x}{|x|}\cdot\nabla$, then
\begin{equation}\label{lem2}
\int_{\Rd}r |\partial_r\phi|^2\leq \int_{\Rd}r |\sigma\cdot\nabla\phi|^2.
\end{equation}
Equality holds if and only if $\phi$ is a radial function.
\end{lem}

\begin{proof}
Here we follow the approach of \cite{DEL} or \cite{V}. Recall that the angular momentum vector $L$ is given by
$$L=-i \nabla\wedge x.$$
A simple calculation shows that 
$$\sigma\cdot\nabla\phi=\left(\sigma\cdot\frac{x}{r}\right)\left(\partial_r-\frac{1}{r}\sigma\cdot L\right)$$
and it can be easily seen  that
$$\int_{\Rd}r\left|\left(\sigma\cdot\frac{x}{|x|}\right)\phi\right|^2=\int_{\Rd}r|\phi|^2.$$
Therefore,
\begin{eqnarray*}
\int_{\Rd} r |\sigma\cdot\nabla\phi|^2&=&\int_{\Rd} r\left|\left(\partial_r-\frac{1}{r}\sigma\cdot L\right)\phi\right|^2\\
&=&\int_{\Rd} r|\partial_r\phi|^2+\int_{\Rd} \frac{1}{r}|\sigma\cdot L\phi|^2-\int_{\Rd} r \partial_r\phi\left(\overline{\frac{1}{r}\sigma\cdot L \phi}\right)\\
&-&\int_{\Rd} r \left(\frac{1}{r}\sigma\cdot L \phi\right) \overline{\partial_r\phi}.
\end{eqnarray*}
Since $\sigma\cdot L$ and $\partial_r$ commute and since $\sigma\cdot L$ is symmetric, we have
\begin{eqnarray*}
-\int_{\Rd} \partial_r(\phi (\overline{\sigma\cdot L \phi}))&=& -\int_{\Rd} \partial_r\phi (\overline{\sigma\cdot L \phi})-\int_{\Rd} \phi \, \overline{\partial_r(\sigma\cdot L \phi)}\\
&=&-\int_{\Rd} r \partial_r\phi\left(\overline{\frac{1}{r}\sigma\cdot L \phi}\right)-\int_{\Rd} r \left(\frac{1}{r}\sigma\cdot L \phi\right) \overline{\partial_r\phi}.
\end{eqnarray*}
Moreover, writing in polar coordinates and integrating by parts we get
\begin{eqnarray*}
-\int_{\Rd} \partial_r(\phi (\overline{\sigma\cdot L \phi}))&=&\int_{\Rd} \frac{2}{r}\phi (\overline{\sigma\cdot L \phi}).
\end{eqnarray*}
In consequence,
\begin{eqnarray*}
\int_{\Rd} r |\sigma\cdot\nabla\phi|^2&=&\int_{\Rd} r|\partial_r\phi|^2+\int_{\Rd} \frac{1}{r}|\sigma\cdot L\phi|^2+\int_{\Rd} \frac{2}{r}\phi (\overline{\sigma\cdot L \phi})\\
&+&\int_{\Rd}\frac{1}{r}|\phi|^2-\int_{\Rd}\frac{1}{r}|\phi|^2\\
&=& \int_{\Rd} r|\partial_r\phi|^2+\int_{\Rd} \frac{1}{r}|(\sigma\cdot L+1)\phi|^2-\int_{\Rd}\frac{1}{r}|\phi|^2.
\end{eqnarray*}
The result follows from the fact 
$$\int_{S^2} |(\sigma\cdot L+1)\phi|^2\geq \int_{S^2} |\phi|^2.$$

Moreover, since Ker($\sigma\cdot L$) contains all radial functions, equality holds in (\ref{lem2}) for $\phi$ radial.
\end{proof}

\begin{lem}\label{inq.for_fi}
Let $\phi:\Rd\to\C$ and $0\leq m <+\infty$. Then
\begin{equation}\label{for_fi}
\int_{\Rd} |\phi|^2\frac{1}{|x|}+\sqrt{1+m^2}\int_{\Rd} |\phi|^2\leq \int_{\Rd}|\partial_r\phi|^2|x|+(1+m^2)\int_{\Rd} |\phi|^2|x|.
\end{equation}
\end{lem}

\begin{proof}
Let us develop the following square
\begin{eqnarray*}
0&\leq&\int_{\Rd}\left||x|^{1/2}\partial_r\phi+\frac{1}{|x|^{1/2}}\phi+\sqrt{1+m^2}|x|^{1/2}\phi\right|^2\\
&=&\int_{\Rd}|x||\partial_r\phi|^2+\int_{\Rd}|\phi|^2\frac{1}{|x|}+(1+m^2)\int_{\Rd}|\phi|^2|x|\\
&+&\int_{\Rd}\partial_r|\phi|^2+\sqrt{1+m^2}\int_{\Rd}|x|\partial_r|\phi|^2+2\sqrt{1+m^2}\int_{\Rd}|\phi|^2\\
&=&\int_{\Rd}|x||\partial_r\phi|^2-\int_{\Rd}|\phi|^2\frac{1}{|x|}+(1+m^2)\int_{\Rd}|\phi|^2|x|-\sqrt{1+m^2}\int_{\Rd}|\phi|^2,
\end{eqnarray*}
which gives the desired inequality. Observe that the last identity holds because
 \begin{eqnarray*}
\int_{\Rd}\partial_r|\phi(x)|^2dx=-\int_{\Rd}\frac{2}{|x|}|\phi|^2
\end{eqnarray*}
and
\begin{eqnarray*}
\int_{\Rd}|x|\partial_r|\phi|^2=-3\int_{\Rd}|\phi|^2,
\end{eqnarray*}
which have been computed writing them in polar coordinates and integrating by parts.
\end{proof}

\begin{remark}\label{oharra}
It is easy to check that for $\phi_0= C r^{-1}e^{-\sqrt{1+m^2}r}$ and $C\in\C$ equality holds in (\ref{for_fi}) if we understand
$$\int_{\Rd}\left(|\partial_r\phi|^2|x|-  |\phi|^2\frac{1}{|x|}\right)dx=\lim_{\delta\to 0}\int_{|x|\geq\delta}\left(|\partial_r\phi|^2|x|-  |\phi|^2\frac{1}{|x|}\right)dx=0.$$
\end{remark}

From Lemma \ref{Lemma2} and Lemma \ref{inq.for_fi} we obtain the following result.

\begin{cor}\label{inq.for_fi_nabla}
Let $\phi:\Rd\to \C^2$ and $0\leq m <+\infty$, then
\begin{equation}\label{for_fi_nabla}
\sqrt{1+m^2}\int_{\Rd} |\phi|^2\leq \int_{\Rd}|\sigma\cdot\nabla\phi|^2|x|+(1+m^2)\int_{\Rd} |\phi|^2|x|-\int_{\Rd} |\phi|^2\frac{1}{|x|}.
\end{equation}
\end{cor}

At this point, we have all the necessary results for proving the main Hardy-Dirac inequality of this section.

\begin{proof}[Proof of Theorem \ref{final_ineq}]
Without loss of generality, let us assume that $\epsilon=1$. This can be done because of the scaling invariance. We will develop the right hand side term in (\ref{final}),
\begin{eqnarray*}
 \int_{\Rd}|x| |( \alpha\cdot\nabla+im\beta\pm \I_4)\psi|^2&=& \int_{\Rd}|x| |(\alpha\cdot\nabla)\psi|^2+ \int_{\Rd}|x| |(im\beta\pm \I_4)\psi|^2\\
 &+&\langle  \alpha\cdot\nabla \psi, (im\beta\pm \I_4) \psi\rangle_{L^2(|x|)}\\
 &+&\langle (im\beta\pm \I_4) \psi,  \alpha\cdot\nabla \psi\rangle_{L^2(|x|)}.
\end{eqnarray*}
The last two terms in the sum can be written and estimated as follows
\begin{eqnarray*}
&&\langle  \alpha\cdot\nabla \psi, (im\beta\pm \I_4) \psi\rangle_{L^2(|x|)}+\langle (im\beta\pm \I_4) \psi,  \alpha\cdot\nabla \psi\rangle_{L^2(|x|)}\\
&=&\langle  \sigma\cdot\nabla \chi, (im\pm \I_4) \phi|x|\rangle_{L^2}+\langle  \sigma\cdot\nabla \phi, (-im\pm \I_4) \chi |x|\rangle_{L^2}\\
&+&\langle (im\pm \I_4) \phi, ( \sigma\cdot\nabla \chi) |x|\rangle_{L^2}+\langle (-im\pm \I_4) \chi, ( \sigma\cdot\nabla \phi) |x|\rangle_{L^2}\\
&=&-\langle  \chi, (im\pm \I_4) ( \sigma\cdot\nabla\phi)|x|\rangle_{L^2}-\langle  \chi, (im\pm \I_4) \phi( \sigma\cdot\nabla|x|)\rangle_{L^2}\\
&-&\langle \phi, (-im\pm \I_4) (  \sigma\cdot\nabla\chi) |x|\rangle_{L^2}-\langle \phi, (-im\pm \I_4) \chi(  \sigma\cdot\nabla |x|)\rangle_{L^2}\\
&+&\langle (im\pm \I_4) \phi, ( \sigma\cdot\nabla \chi) |x|\rangle_{L^2}+\langle (-im\pm \I_4) \chi, ( \sigma\cdot\nabla \phi) |x|\rangle_{L^2}\\
&=&-\left\langle  \chi, (im\pm \I_4)\left( \sigma\cdot\frac{x}{|x|}\right)\phi\right\rangle_{L^2}-\left\langle \phi, (-im\pm \I_4) \left( \sigma\cdot\frac{x}{|x|}\right)\chi\right\rangle_{L^2}\\
&=&2\Re\int_{\Rd}\chi\cdot\overline{(im\pm \I_4)\left( \sigma\cdot\frac{x}{|x|}\right)\phi}.
\end{eqnarray*}
Notice that we have used integration by parts and the fact that $i\sigma\cdot\nabla$ is symmetric in $L^2$ to rewrite the expression. By using the Cauchy-Schwarz inequality we get the estimate
\begin{align}\label{bound}
&2\Re\int_{\Rd}\chi\cdot\overline{(im\pm \I_4)\left( \sigma\cdot\frac{x}{|x|}\right)\phi}\leq 2 \left|\int_{\Rd}\chi\cdot\overline{(im\pm \I_4)\left( \sigma\cdot\frac{x}{|x|}\right)\phi}\right|\\
&\leq\sqrt{1+m^2}\|\chi\|_{L^2}\|\phi\|_{L^2}\leq\sqrt{1+m^2}\|\chi\|_{L^2}^2+\sqrt{1+m^2}\|\phi\|_{L^2}^2.\notag
\end{align}
Now by inequality (\ref{for_fi_nabla}) we conclude
\begin{eqnarray*}
&&|\langle  \alpha\cdot\nabla \psi, (im\beta\pm \I_4) \psi\rangle_{L^2(|x|)}+\langle (im\beta\pm \I_4) \psi,  \alpha\cdot\nabla \psi\rangle_{L^2(|x|)}|\\
&=&2 \left|\int_{\Rd}\chi\cdot\overline{(im\pm \I_4)\left( \sigma\cdot\frac{x}{|x|}\right)\phi}\right|\\
&\leq&\int_{\Rd}|\sigma\cdot\nabla\phi|^2|x|+(1+m^2)\int_{\Rd} |\phi|^2|x|-\int_{\Rd} |\phi|^2\frac{1}{|x|}\\
&+&\int_{\Rd}|\sigma\cdot\nabla\chi|^2|x|+(1+m^2)\int_{\Rd} |\chi|^2|x|-\int_{\Rd} |\chi|^2\frac{1}{|x|}.
\end{eqnarray*}
Thus we can estimate from below the right hand side of (\ref{final}) as
\begin{align}\label{ali}
 &\int_{\Rd}|x| |( \alpha\cdot\nabla+im\beta\pm \I_4)\psi|^2\notag\\
 &\geq \int_{\Rd}|x| |(\sigma\cdot\nabla)\phi|^2+\int_{\Rd}|x| |(\sigma\cdot\nabla)\chi|^2\\
 &+ (1+m^2)\int_{\Rd} |\phi|^2|x|+ (1+m^2)\int_{\Rd} |\chi|^2|x|\notag\\
 &-\int_{\Rd}|\sigma\cdot\nabla\phi|^2|x|-(1+m^2)\int_{\Rd} |\phi|^2|x|+\int_{\Rd} |\phi|^2\frac{1}{|x|}\notag\\
 &-\int_{\Rd}|\sigma\cdot\nabla\chi|^2|x|-(1+m^2)\int_{\Rd} |\chi|^2|x|+\int_{\Rd} |\chi|^2\frac{1}{|x|}\notag\\
 &=\int_{\Rd} |\psi|^2\frac{1}{|x|}.\notag
 \end{align}

In order to see the sharpness of the inequality we will recover the $\epsilon$. If we choose
 $$\chi=\lambda\left(\sigma\cdot\frac{x}{|x|}\right)\phi, \quad \lambda\in \C,$$
 then identity holds in (\ref{bound}). Now recall Remark \ref{oharra}. If we take $\phi_0^{\epsilon,m}=C r^{-1}e^{-\sqrt{\epsilon^2+m^2} r}$, $C\in\C^2$,  we get an identity in (\ref{ali}) in the following sense
$$\int_{\Rd}\left(|x| |( \alpha\cdot\nabla+im\beta\pm \I_4)\psi_0^{\epsilon,m}|^2- |\psi_0^{\epsilon,m}|^2\frac{1}{|x|}\right)dx =0$$ 
 where $\psi_0^{\epsilon,m}$ is defined as in Remark \ref{ohar2}. This completes the proof.
\end{proof}

The following lemmas are also required for the proof of Theorem \ref{teorema}.
\begin{lem}\label{Lemma1}
Let $\phi:\Rd\to \C^2$. Then
\begin{equation}\label{lem1}
\int_{\Rd} |\phi|^2\leq \frac{2}{3}\left(\int_{\Rd}r |\phi|^2\right)^{1/2}\left(\int_{\Rd}r |\partial_r\phi|^2\right)^{1/2}.
\end{equation}
Equality holds for $\phi=Ce^{-\lambda r}$, where $C\in\C$ and $\lambda>0$.
\end{lem}

\begin{proof}
From the fundamental theorem of calculus and following the same approach of \cite{DV} we have 
$$|\phi(r\omega)|^2=\Re\left(\int_r^{\infty}-2 \overline{\phi(t \omega)}\, (\omega\cdot\nabla\phi(t \omega))dt\right).$$
Then
\begin{eqnarray*}
\int_{\Rd}|\phi|^2 dx&=& \int_{S^2}d\omega\int_0^{\infty}|\phi(r\omega)|^2 r^2 dr\\
&=&-2\Re\left(\int_{S^2}d\omega\int_0^{\infty}r^2 dr\int_r^{\infty}\overline{\phi(t \omega)}\, (\omega\cdot\nabla\phi(t \omega))dt\right)\\
&=&-2\Re\left(\int_{S^2}d\omega\int_0^{\infty}\overline{\phi(t \omega)}\, (\omega\cdot\nabla\phi(t \omega))\int_0^t r^2 dr dt\right)\\
&=&-\frac{2}{3}\Re\left(\int_{\Rd}\overline{\phi(x)}(\partial_r \phi(x))rdx\right)\\
&\leq&\frac{2}{3}\left(\int_{\Rd}r|\phi|^2\right)^{1/2}\left(\int_{\Rd}r|\partial_r\phi|^2\right)^{1/2}.
\end{eqnarray*}
Let us check that equality holds for $\phi=Ce^{-\lambda r}$. On the one hand,
\begin{eqnarray*}
\int_{\Rd}|\phi|^2&=&4\pi C\int_0^{\infty}e^{-2\lambda r} r^2 dr=\frac{4\pi C}{(2\lambda)^3}\int_0^{\infty}e^{- r} r^2 dr.
\end{eqnarray*}
On the other hand,
$$\int_{\Rd}r|\partial_r\phi|^2=\lambda^2\int_{\Rd}|\phi|^2r.$$
Finally,
$$\int_{\Rd}|\phi|^2r=4\pi C \int_0^{\infty}e^{-2\lambda r} r^3 dr=\frac{3C}{(2\lambda)^4}4\pi\int_0^{\infty}e^{- r} r^2 dr.$$
Therefore, equality holds.
\end{proof}

From Lemma \ref{Lemma2} and Lemma \ref{Lemma1} we conclude the following result.
\begin{lem}\label{Lemma3}
Let $\phi:\Rd\to \C^2$. Then
\begin{equation}\label{lem3}
\int_{\Rd} |\phi|^2\leq \frac{2}{3}\left(\int_{\Rd}r |\phi|^2\right)^{1/2}\left(\int_{\Rd}r |\sigma \cdot \nabla\phi|^2\right)^{1/2}.
\end{equation}
Moreover, equality holds for $\phi= Ce^{-\lambda r}$, where $C\in\C^2$ and $\lambda>0$.
\end{lem}

Finally we use inequality (\ref{lem3}) for proving the next result.

\begin{pro}\label{proImproved}
 For any $m\geq 0$ finite,
 \begin{equation}\label{ineq.2.improved}
\frac{8}{9} \int_{\Rd} |x||i \alpha\cdot\nabla\psi|^2 \leq \int_{\Rd} |x||(i \alpha\cdot\nabla-m\beta\pm i)\psi|^2.
 \end{equation}
 Equality holds for $\phi= Ce^{-\lambda r}$ and $\chi=i\sigma\cdot\frac{x}{|x|}\phi$ where $C\in\C^2$ and $\lambda=3 \sqrt{1+m^2}$.
 \end{pro}

\begin{remark}
The fact that inequality (\ref{ineq.2.improved}) is sharp, in the sense that the constant can not be improved, confirms that the proof of (\ref{final}) is not an immediate consequence of the Hardy inequality
$$ \int_{\Rd}\frac{1}{|x|}|\psi|^2\leq  \int_{\Rd}|x||i \alpha\cdot\nabla\psi|^2.$$
\end{remark}

\begin{proof}
We follow the approach of Theorem \ref{final_ineq}. Therefore, as we have seen previously we have that
\begin{eqnarray*}
&& \int_{\Rd}|x| |(i \alpha\cdot\nabla-m\beta\pm i)\psi|^2\\
&=&  \int_{\Rd}|x| |(\sigma\cdot\nabla)\phi|^2+\int_{\Rd}|x| |(\sigma\cdot\nabla)\chi|^2\\
 &+& (1+m^2)\int_{\Rd} |\phi|^2|x|+ (1+m^2)\int_{\Rd} |\chi|^2|x|\\
 &+&\left\langle  \chi, (-m\pm i)\left(i \sigma\cdot\frac{x}{|x|}\right)\phi\right\rangle_{L^2}+\left\langle \phi, (m\pm i) \left(i \sigma\cdot\frac{x}{|x|}\right)\chi\right\rangle_{L^2}.
\end{eqnarray*}
From inequality (\ref{lem3}) we obtain
\begin{eqnarray*}
\sqrt{1+m^2}\|\phi\|_{L^2}^2&\leq& \frac{2}{3}\sqrt{1+m^2}\left(\int_{\Rd}|x| |\phi|^2\right)^{1/2}\left(\int_{\Rd}|x| |\sigma \cdot \nabla\phi|^2\right)^{1/2}\\
&\leq& (1+m^2)\int_{\Rd}|x| |\phi|^2+\frac{1}{9}\int_{\Rd}|x| |\sigma \cdot \nabla\phi|^2.
\end{eqnarray*}
By the last inequality and (\ref{bound}) we conclude
\begin{eqnarray*}
 \int_{\Rd}|x| |(i \alpha\cdot\nabla-m\beta\pm i)\psi|^2&\geq&  \int_{\Rd}|x| |(\sigma\cdot\nabla)\phi|^2+\int_{\Rd}|x| |(\sigma\cdot\nabla)\chi|^2\\
 &+& (1+m^2)\int_{\Rd} |\phi|^2|x|+ (1+m^2)\int_{\Rd} |\chi|^2|x|\\
 &-& (1+m^2)\int_{\Rd}|x| |\phi|^2-\frac{1}{9}\int_{\Rd}|x| |\sigma \cdot \nabla\phi|^2\\
 &-& (1+m^2)\int_{\Rd}|x| |\chi|^2-\frac{1}{9}\int_{\Rd}|x| |\sigma \cdot \nabla\chi|^2\\
 &=&\frac{8}{9}\int_{\Rd}|x| |(i \alpha\cdot\nabla)\psi|^2.
 \end{eqnarray*}
The equality case can be easily checked.
\end{proof}

\section{Proof of Theorem \ref{teorema}}\label{sec:pde} 

 The proof of the first part of the theorem is divided into five steps.

\emph{Step 1.} We start by proving that for all $f\in L^2(1+|x|)$ there exists a unique function $\psi$ in 
$$\widetilde{\mathcal{D}}=\{\psi\in L^2(1+|x|): H\psi\in L^2(1+|x|)\}$$ 
such that 
\begin{equation}\label{eq.}
(H + i )\psi=f.
\end{equation}
We rewrite (\ref{eq.}) as 
$$(i\alpha\cdot \nabla-m\beta - i)\psi+\V\psi = f,$$ 
or equivalently,
$$|x|^{1/2}(i\alpha\cdot \nabla -m\beta- i)|x|^{1/2}\frac{1}{|x|^{1/2}}\psi+ |x|\V \frac{1}{|x|^{1/2}}\psi = |x|^{1/2}f.$$
Denote $w=|x|^{-1/2}\psi,\, F=|x|^{1/2}f \;\text{and}\; \widetilde{K}=|x|^{1/2} (i\alpha\cdot\nabla-m\beta - i )|x|^{1/2}$.
Hence, we have 
$$\widetilde{K}w+|x|\V w=F.$$ 
Let $K$ be the inverse of $\widetilde{K}$, that is, $K=|x|^{-1/2}(i\alpha\cdot\nabla-m\beta - i )^{-1} |x|^{-1/2}$, and apply $K$ to the previous equation. Thus we get
$$(\I+K|x|\V)w=KF.$$
This equation has a unique solution if $\|K|x|\V \|<1$, that is, if
$$\sup_{F\neq 0} \frac{\|K|x|\V F\|_{L^2}}{\|F\|_{L^2}}< 1.$$
Let us check that for all $F\neq 0$ we have
$$\int_{\Rd}|K|x|\V F|^2\leq \nu \int_{\Rd}|F|^2, \quad \nu<1.$$
Replacing the expression of $K$ we get
$$\int_{\Rd}\frac{1}{|x|}\left|(i\alpha\cdot\nabla -m\beta- i )^{-1}\frac{1}{|x|^{1/2}}|x|\V F\right|^2\leq \nu\int_{\Rd} |F|^2, \quad \nu<1.$$
Now taking into account  the definition of $F$ we obtain
\begin{equation}\label{3ineq.4}
\int_{\Rd}\frac{1}{|x|}|(i\alpha\cdot\nabla -m\beta- i )^{-1}|x|\V f|^2\leq \nu\int_{\Rd}|x| |f|^2, \quad \nu<1.
\end{equation}
In the previous section we have proved the inequality 
\begin{equation*}
\int_{\Rd}\frac{1}{|x|}|(i \alpha\cdot\nabla-m\beta- i)^{-1}f|^2 \leq \int_{\Rd}|x||f|^2.
\end{equation*}
The same estimate holds for $|x|\V f$, that is,
$$\int_{\Rd}\frac{1}{|x|}|(i \alpha\cdot\nabla-m\beta- i)^{-1}|x|\V f|^2 \leq \int_{\Rd}|x|||x|\V f|^2.$$
The inequality is well defined because $f\in L^2(1+|x|)$ and $\sup_{x\in\Rd\backslash \{0\}}|x||\V|<1$, therefore, $|x|\V f\in L^2(1+|x|)$ . In consequence, inequality (\ref{3ineq.4}) holds if $\V$ satisfies that
$$\int_{\Rd}|x|||x|\V f|^2\leq \nu  \int_{\Rd}|x| |f|^2, \quad \nu<1,$$
which is true by the assumption $\sup_{x\in\Rd\backslash \{0\}}||x|\V(x)|< 1$. In consequence, we have proved that there exists a solution of (\ref{eq.}) and it is unique because so is $w$. 

It can be proved in the same way that for all $f\in L^2(1+|x|)$ there exists a unique function $\psi$ in 
$\widetilde{\mathcal{D}}$ such that $(H - i )\psi=f$. We just have to change the sign of $i$ in the computations. In other words, we have proved that Ran$(H\pm i)=L^2(1+|x|)$. 

Let us now show that the solution $\psi$ of (\ref{eq.}) is in $\widetilde{\mathcal{D}}$. We write the solution of $(\I+K|x|\V)w=KF$ by using a Neumann series as
$$w=(\I+K|x|\V)^{-1}KF=\sum_{j=0}^{\infty}(-1)^j(K|x|\V)^j(KF).$$
Since we have denoted $\psi=|x|^{1/2} w$, then
$$ \psi=\sum_{j=0}^{\infty}(-1)^j |x|^{1/2}(K|x|\V)^j \frac{1}{|x|^{1/2}}(i \alpha\cdot\nabla -m\beta- i)^{-1} \frac{F}{|x|^{1/2}}.$$
Equivalently,
\begin{eqnarray*}
\psi&=&(i \alpha\cdot\nabla-m\beta - i)^{-1} f - (i \alpha\cdot\nabla-m\beta- i)^{-1}\V (i \alpha\cdot\nabla-m\beta- i)^{-1}f\\
&+&(i \alpha\cdot\nabla-m\beta- i)^{-1}\V(i \alpha\cdot\nabla-m\beta- i)^{-1}\V (i \alpha\cdot\nabla-m\beta- i)^{-1}f+\dots\\
&=&\sum_{j=0}^{\infty}(-1)^j[(i \alpha\cdot\nabla-m\beta- i)^{-1}\V]^j(i \alpha\cdot\nabla-m\beta- i)^{-1}f.
\end{eqnarray*}
Since $f\in L^2(1+|x|)$, in particular, $f\in L^2(|x|)$. By inequality (\ref{final}) and $\epsilon=1$, 
$$(i\alpha\cdot\nabla-m\beta- i )^{-1}f\in L^2(|x|^{-1}).$$
Since $\sup_{x\in\Rd\backslash \{0\}}|x||\V(x)|<1$, then
$$\V(i\alpha\cdot\nabla-m\beta- i )^{-1}f\in L^2(|x|).$$
Again by inequality (\ref{final}) we get that the second term of $\psi$ is also in $L^2(|x|^{-1})$. The same thing happens with the rest of the terms in the sum. Observe that the series is convergent because the constants of the inequalities we have used are less than or equal to 1.

It is evident that $\psi\in L^2(|x|^{-1})$ is equivalent to $\frac{\psi}{|x|}\in L^2(|x|)$. Then, by the assumption on $\V$, we obtain that $\V\psi\in L^2(|x|)$. We claim that
\begin{equation}\label{3bitarte}
(i\alpha\cdot\nabla-m\beta- i )\psi\in L^2(|x|).
\end{equation}
This is because 
$$(i\alpha\cdot\nabla-m\beta- i )\psi+\V\psi=f$$
and $f,\V\psi\in L^2(|x|)$. From (\ref{3bitarte}) and by inequality (\ref{ineq.2.improved}), $\psi\in L^2(|x|)$. Therefore, we have seen that $\psi\in L^2(|x|)\cap L^2(|x|^{-1})$. Thus by interpolation we obtain that $\psi\in L^2(\Rd, \C^4)$. It remains to prove that $H\psi\in  L^2(1+|x|)$. However, this is true because we can write
$$H\psi=(i\alpha\cdot\nabla-m\beta+\V)\psi=f-i \psi,$$
and we know that $f,\psi\in  L^2(1+|x|)$.

Following the same approach, it can be proved that the solution $\psi$ of $(H - i )\psi=f$ is in $\widetilde{\mathcal{D}}$.

\emph{Step 2.} We are going to see that $H$ defined on $\widetilde{\mathcal{D}}$ is symmetric. So we have to see that 
$$\forall \psi_1,\psi_2\in \widetilde{\mathcal{D}},\quad \langle H\psi_1,\psi_2\rangle=\langle \psi_1,H\psi_2\rangle.$$
Since $\psi_1, \psi_2\in \widetilde{\mathcal{D}}$ then $\psi_1,\psi_2, H\psi_1, H\psi_2\in  L^2(1+|x|)$. Let
$$f_1=(H+i)\psi_1 \quad\text{and} \quad f_2=(H-i)\psi_2.$$
Those functions are in $ L^2(1+|x|)$. Hence,
$$\langle H\psi_1,\psi_2\rangle=\langle (H+i)\psi_1,\psi_2\rangle-\langle i\psi_1, \psi_2\rangle=\langle f_1,\psi_2\rangle-\langle i\psi_1, \psi_2\rangle.$$
Now from Step 1, the solutions $\psi_1$ and $\psi_2$ of $(H+i)\psi_1=f_1$ and $(H-i)\psi_2=f_2$, respectively, can be written by means of a Neumann series. That is, 
\begin{eqnarray*}
 \psi_1&=&\sum_{j=0}^{\infty}(-1)^j[(i \alpha\cdot\nabla-m\beta- i)^{-1}\V]^j(i \alpha\cdot\nabla-m\beta- i)^{-1}f_1
\end{eqnarray*}
and
\begin{eqnarray*}
 \psi_2&=&\sum_{j=0}^{\infty}(-1)^j[(i \alpha\cdot\nabla-m\beta+ i)^{-1}\V]^j(i \alpha\cdot\nabla-m\beta+ i)^{-1}f_2.
\end{eqnarray*}
Observe that 
$$(i\alpha\cdot\nabla - m \beta \pm i)^{-1}= (\Delta +m+1)^{-1}(i\alpha\cdot\nabla - m \beta \mp i).$$
The first term on the right hand side is symmetric. In the second term, notice that $i\alpha\cdot\nabla - m \beta$ is symmetric but $\mp i$ is antisymmetric, thus the sign of this part changes when we pass it from one side to the other in the following inner product. Moreover, $\V$ is symmetric by assumption. Hence,
\begin{eqnarray*}
\langle f_1,\psi_2\rangle&=&\left\langle f_1,\sum_{j=0}^{\infty}(-1)^j[(i \alpha\cdot\nabla-m\beta+ i)^{-1}\V]^j(i \alpha\cdot\nabla-m\beta+ i)^{-1}f_2\right\rangle\\
&=&\sum_{j=0}^{\infty}(-1)^j\left\langle f_1, [(i \alpha\cdot\nabla-m\beta+ i)^{-1}\V]^j(i \alpha\cdot\nabla-m\beta+ i)^{-1}f_2\right\rangle\\
&=&\sum_{j=0}^{\infty}(-1)^j\left\langle [(i \alpha\cdot\nabla-m\beta- i)^{-1}\V]^j(i \alpha\cdot\nabla-m\beta- i)^{-1}f_1, f_2\right\rangle\\
&=&\left\langle \sum_{j=0}^{\infty}(-1)^j[(i \alpha\cdot\nabla-m\beta- i)^{-1}\V]^j(i \alpha\cdot\nabla-m\beta- i)^{-1}f_1, f_2\right\rangle\\
&=&\langle \psi_1, f_2\rangle.
\end{eqnarray*}
Therefore,
\begin{eqnarray*}
\langle H\psi_1,\psi_2\rangle&=&\langle f_1,\psi_2\rangle-\langle i\psi_1, \psi_2\rangle=\langle \psi_1,f_2\rangle-\langle i\psi_1, \psi_2\rangle\\
&=&\langle \psi_1,(H-i)\psi_2\rangle-\langle i\psi_1, \psi_2\rangle=\langle \psi_1,H\psi_2\rangle.
\end{eqnarray*}

\emph{Step 3.} Let us prove that for all $f\in L^2(1+|x|)$ we have
\begin{equation}\label{efe}
\|\psi\|_{L^2}\leq \|f\|_{L^2}.
\end{equation}
In the first step we have seen that for $f\in L^2(1+|x|)$ there exists a function such that $(H+i)\psi=f$. Multiply this equation by $\overline{\psi}$, integrate over $\Rd$ and take the imaginary parts to get
$$\Im\langle (H+i)\psi, \psi\rangle=\Im\langle f, \psi\rangle.$$
As we have proved that $H$ is symmetric, then $\langle H\psi, \psi\rangle$ is real. Hence,
$$\langle \psi, \psi\rangle=\Im\langle (H+i)\psi, \psi\rangle=\Im\langle f, \psi\rangle\leq \|f\|_{L^2}\|\psi\|_{L^2},$$
which completes the proof of this part.

\emph{Step 4.} At this point we have all the ingredients to prove (\ref{theorem}). From Step 1 we know that for every $f$ in $L^2(1+|x|)$ there exists a function $\psi$ in $\widetilde{\mathcal{D}}$ such that $ (H+ i )\psi=f$.
Since $L^2(1+|x|)$ is dense in $L^2(\Rd,\C^4)$, for each $f\in L^2(\Rd,\C^4)$ there exists a sequence $\{f_n\}_{n\in \N}$ in  $L^2(1+|x|)$ such that 
$$\lim_{n\rightarrow \infty}\|f_n-f\|_{L^2}=0.$$
Then, $f_n$ is a Cauchy sequence in $L^2(\Rd,\C^4)$. Now since $f_n\in L^2(1+|x|)$, by Step 1, $(H+ i )\psi_n=f_n$ for $\psi_n\in \widetilde{\mathcal{D}}$.
Notice that $\psi_n$ is a Cauchy sequence in $L^2(\Rd,\C^4)$ because by Step 3
 $$\|\psi_n-\psi_m\|_{L^2}\leq \|f_n-f_m\|_{L^2}, \quad n, m\in \N.$$
Thus there exists a function $\psi\in L^2(\Rd,\C^4)$ such that
 $$\lim_{n\rightarrow \infty}\|\psi_n-\psi\|_{L^2}=0.$$
Since $H\psi_n$ converges to $H\psi$ in the distributional sense and $H\psi_n=f_n-i\psi_n$ converges to $f-i\psi$ in $L^2(\Rd,\C^4)$, $H\psi=f-i\psi$.

We have that for a function $f_n\in L^2(1+|x|)$ there exists a function $\psi_n$ such that $(H+ i)\psi_n=f_n$. Thus, in particular, we have the weaker identity
$$\int_{\Rd}(H+ i)\psi_n \cdot \overline{\varphi}=\int_{\Rd}f_n \cdot \overline{\varphi},$$ 
where each of the four components of $\varphi$ is in the Schwartz class. Since we have seen that $H$ is symmetric,
$$\int_{\Rd}f_n \cdot \overline{\varphi}=\int_{\Rd}(H+ i)\psi_n \cdot \overline{\varphi}=\int_{\Rd}\psi_n \cdot \overline{(H- i)\varphi}.$$
We take the limit when $n$ tends to infinity to get
$$\int_{\Rd}f \cdot \overline{\varphi}=\lim_{n\rightarrow \infty}\int_{\Rd}f_n \cdot \overline{\varphi}=\lim_{n\rightarrow \infty}\int_{\Rd}\psi_n \cdot \overline{(H- i)\varphi}=\int_{\Rd}\psi \cdot \overline{(H- i)\varphi},$$
and we obtain the desired result.

The equation (\ref{theorem2}) is proved following the same approach. We only have to change $+i$ by $-i$.

 \emph{Step 5.} For the uniqueness of the solutions we will need the following point (i). In (i) we will prove that for all $f\in  L^2(\Rd,\C^4)$,
$$\|\psi\|_{L^2}\leq \|f\|_{L^2}.$$
Hence, if $f=0$ then $\psi=0$. Therefore, there is uniqueness in Step 4. Once we have this, since $\psi\in L^2(\Rd,\C^4)$ and by using the same argument we conclude that there exists a unique $\varphi \in L^2(\Rd,\C^4)$ such that $(H+i)\varphi=\psi$. In consequence,
$$\int_{\Rd}|\psi|^2=\int_{\Rd}\psi \cdot \overline{(H+i)\varphi}=\int_{\Rd}f \cdot \overline{\varphi}\leq \|f\|_{L^2}\|\varphi\|_{L^2}\leq\|f\|_{L^2}\|\psi\|_{L^2}.$$
Thus we conclude the uniqueness of solution of (\ref{theorem}). The same argument follows for the uniqueness of (\ref{theorem2}).

We now proceed to prove the second part of the theorem.
 
 (i) For $\psi_n$ and $f_n$ as in Step 4 we have
 $$\|\psi_n\|_{L^2}\leq\|f_n\|_{L^2}$$
 from Step 3. Passing to the limit as $n$ goes to infinity we get
  $$\|\psi\|_{L^2}\leq\|f\|_{L^2}.$$
(ii) Let us check now that 
\begin{equation*}
\int_{\Rd}\frac{|\psi|^2}{|x|}\leq c \int_{\Rd}|f|^2,
\end{equation*}
 where $c$ is a positive constant. We define a smooth cut-off function $\eta$ such that  $0\leq \eta(x)\leq 1$ and
$$
\eta(x):=
\left\{
\begin{array}{rl}
1 & \mbox{if } |x|\leq 1, \\
0 & \mbox{if } |x|> 2.
\end{array}
\right .
$$
Take
$$\widetilde{f}=(H\pm i)(\eta\psi)=(-i\alpha\cdot\nabla+m\beta-\V \pm i)(\eta\psi)=\eta(H\pm i)\psi-(i\alpha\cdot\nabla\eta)\psi.$$
First notice that $\widetilde{f}$ is in $L^2(1+|x|)$. The first term because $f=(H\pm i)\psi$ is in $L^2(\Rd,\C^4)$, and therefore, $\eta(H\pm i)\psi\in L^2(1+|x|)$. And the second term because, since $f\in L^2(\Rd,\C^4)$, its corresponding $\psi$ belongs to  $L^2(\Rd,\C^4)$. In consequence, $(i\alpha\cdot\nabla\eta)\psi\in L^2(\Rd,\C^4)$, and since it is compactly supported, $(i\alpha\cdot\nabla\eta)\psi\in L^2(1+|x|)$.
 
Since $\widetilde{f}\in L^2(1+|x|)$ and by inequality (\ref{final}) we conclude that $\eta\psi \in L^2(|x|^{-1})$. Therefore,
 \begin{eqnarray*}
 \int_{\Rd}|\psi|^2\frac{1}{|x|}&=& \int_{|x|\leq 1}|\psi|^2\frac{1}{|x|}+ \int_{|x|>1}|\psi|^2\frac{1}{|x|}\\
 &\leq& \int_{|x|\leq 1}|\eta\psi|^2\frac{1}{|x|}+ \int_{|x|>1}|\psi|^2\\
 &\leq&  \int_{\Rd}|\widetilde{f}|^2|x|+ \int_{\Rd}|\psi|^2\\
 &\leq&  c\left(\int_{\Rd}|\eta f|^2|x|+ \int_{\Rd}|x||(i\alpha\cdot\nabla \eta)\psi|^2 \right)+\int_{\Rd}|f|^2\\
 &\leq& c \int_{|x|\leq 2}|f|^2+ c \int_{|x|\leq 2}|\psi|^2+\int_{\Rd}|f|^2\\
 &\leq& c\int_{\Rd}|f|^2.
 \end{eqnarray*}
 
 (iii) We start with 
 \begin{eqnarray*}
&& \int_{|x|\leq 1}|x||i\alpha\cdot\nabla\psi|^2\\
&=& \int_{|x|\leq 1}|x||(i\alpha\cdot\nabla-m\beta+\V\pm i)\psi+(m\beta-\V\mp i)\psi|^2\\
&\leq&c \int_{|x|\leq 1}|x||(H\pm i)\psi|^2+c (m^2+1) \int_{|x|\leq 1}|x||\psi|^2+c \int_{|x|\leq 1}|x|\left|\V\psi\right|^2\\
&\leq& c \int_{\Rd}|f|^2.
 \end{eqnarray*} 
The last estimate comes from the following. First note that 
\begin{eqnarray*}
 \int_{|x|\leq 1}|x||(H\pm i)\psi|^2&=&\int_{|x|\leq 1}|x||f|^2 \leq \int_{\Rd}|f|^2.
 \end{eqnarray*} 
Secondly, from (i) we have that
$$(m^2+1) \int_{|x|\leq 1}|x||\psi|^2\leq c \int_{|x|\leq 1}|\psi|^2\leq c \int_{\Rd}|f|^2.$$
Finally,
$$\int_{|x|\leq 1}|x|\left|\V\psi\right|^2\leq\int_{|x|\leq 1}|x|\left|\frac{1}{|x|}\psi\right|^2\leq\int_{|x|\leq 1}\frac{1}{|x|}|\psi|^2\leq c\int_{\Rd}|f|^2,$$
where the last estimate holds because of (ii).

(iv) Since $f\in L^2(\Rd,\C^4)$, its corresponding $\psi$ is also in $L^2(\Rd,\C^4)$. Taking the Fourier transform it is easy to check that  the identity 
$$\int_{\Rd} (i\alpha\cdot\nabla)\psi\cdot\overline{\sqrt{-\Delta}(i\alpha\cdot\nabla)^{-1}\psi}=\int_{\Rd}\psi \cdot\overline{\sqrt{-\Delta}\psi}$$
holds when both terms are finite. Define $\phi=\sqrt{-\Delta}(i\alpha\cdot\nabla)^{-1}\psi$. Then
\begin{eqnarray*}
\left|\int_{\Rd} (i\alpha\cdot\nabla)\psi\cdot\overline{\phi}\right|&=&\left|\int_{\Rd}(i\alpha\cdot\nabla+\V)\psi\cdot\overline{\phi}-\V\psi\cdot\overline{\phi}\right|\\
&\leq& \left|\int_{\Rd} (i\alpha\cdot\nabla+\V)\psi\cdot\overline{\phi}\right|+\left|\int_{\Rd}\V\psi\cdot\overline{\phi}\right|\\
&\leq& \left|\int_{\Rd} (i\alpha\cdot\nabla+\V)\psi\cdot\overline{\phi}\right|+\int_{\Rd}\frac{1}{|x|}|\psi\cdot\overline{\phi}|\\
&\leq& \left(\int_{\Rd} |(i\alpha\cdot\nabla+\V)\psi|^2\right)^{1/2}\left(\int_{\Rd}|\phi|^2\right)^{1/2}\\
&+&\left(\int_{\Rd}|\psi|^2|x|^{-1}\right)^{1/2}\left(\int_{\Rd}|\phi|^2|x|^{-1}\right)^{1/2}.
\end{eqnarray*}  
Let us estimate the first term on the right side. The first integral can be estimated as follows
\begin{eqnarray*}
\int_{\Rd} |(i\alpha\cdot\nabla+\V)\psi|^2&\leq&c \int_{\Rd}|(H+ i)\psi|^2+c(m^2+1)\int_{\Rd}|\psi|^2\\
&\leq&c \|f\|^2_{L^2}+c  \|f\|^2_{L^2}\leq c  \|f\|^2_{L^2}.
\end{eqnarray*}  
Here we have used (i) and the fact that 
$$f=(H+ i)\psi=(i\alpha\cdot\nabla+\V-m\beta+i)\psi.$$
On the other hand,  for the second term we denote by $\mathcal{F}(g)$ the Fourier transform of a function $g$. Thus,
\begin{eqnarray*}
\int_{\Rd}|\phi|^2=\int_{\Rd}|\sqrt{-\Delta}(i\alpha\cdot\nabla)^{-1}\psi|^2=\int_{\Rd}|\mathcal{F}(\sqrt{-\Delta}(i\alpha\cdot\nabla)^{-1}\psi)|^2.
\end{eqnarray*}  
It is easy to check that 
$$\int_{\Rd}|\mathcal{F}(\sqrt{-\Delta}(i\alpha\cdot\nabla)^{-1}\psi)|^2=\int_{\Rd}|\mathcal{F}(\psi)|^2,$$
just notice that
$$\mathcal{F}((i\alpha\cdot\nabla)^{-1}g)=\frac{1}{|\xi|^2}\mathcal{F}((i\alpha\cdot\nabla) g).$$
Therefore,
$$\int_{\Rd}|\phi|^2=\int_{\Rd}|\mathcal{F}(\psi)|^2=\int_{\Rd}|\psi|^2\leq \|f\|^2_{L^2}.$$
The first integral of the second term is bounded by $c\|f\|_{L^2}$ because of (ii). And the second integral because
$$\left(\int_{\Rd}|\phi|^2|x|^{-1}\right)^{1/2}\leq c \left(\int_{\Rd}|\psi|^2|x|^{-1}\right)^{1/2}\leq c \|f\|_{L^2}.$$
Observe that the first inequality holds because $|x|^{-1}$ is an $A_2$ weight and $\sqrt{-\Delta}(i\alpha\cdot\nabla)^{-1}$ can be expressed by means of the Riesz transforms. In consequence,
$$\int_{\Rd}\psi \cdot\overline{\sqrt{-\Delta}\psi}=\int_{\Rd} i\alpha\cdot\nabla\psi\cdot\overline{\phi}\leq c \|f\|^2_{L^2}.$$
Therefore, taking into account what we have proved until now and since $\psi\in L^2(\Rd,\C^4)$ we get that
$$\|\psi\|_{H^{1/2}}\leq c \|f\|_{L^2}.$$

(v) Let $f^1, f^2\in L^2(\Rd,\C^4)$ and take sequences $f_n^1, f_n^2 \in L^2(1+|x|)$ such that
$$f^1=\lim_{n\rightarrow \infty}f_n^1 \quad \text{and} \quad f^2=\lim_{n\rightarrow \infty}f_n^2$$
in $L^2(\Rd,\C^4)$. We take
$$(H+ i)\psi_n^1=f_n^1\quad \text{and}\quad (H- i)\psi_n^2=f_n^2$$
and follow the approach of Step 2. Then,
$$\int_{\Rd}(H+ i)\psi_n^1\cdot\overline{\psi_n^2}=\int_{\Rd}\psi_n^1\cdot\overline{(H- i)\psi_n^2}.$$
We let $n$ tend to infinity and obtain the desired conclusion. Taking
$$(H- i)\widetilde{\psi}_n^1=f_n^1\quad \text{and}\quad (H+ i)\widetilde{\psi}_n^2=f_n^2$$
and following the same approach we get
$$\int_{\Rd}(H- i)\widetilde{\psi}_n^1\cdot\overline{\widetilde{\psi}_n^2}=\int_{\Rd}\widetilde{\psi}_n^1\cdot\overline{(H+ i)\widetilde{\psi}_n^2}.$$
This completes the proof of the theorem.

\section{Proof of Theorem \ref{3thm}}\label{sec:s_a}

We use the basic criterion for  self-adjointness and Theorem \ref{teorema} for the proof in this section. First of all, we will see that 
$$\text{Ran}(H\pm i)=L^2(\Rd,\C^4).$$
In Theorem \ref{teorema} we have already proved that for all $f\in L^2(\Rd,\C^4)$ there exists a function $\psi\in \mathcal{D}(H)$ such that $(H+ i)\psi=f$. Analogously, for all $f\in L^2(\Rd,\C^4)$ there exists another function $\psi\in \mathcal{D}(H)$ such that $(H- i)\psi=f$.

Let us show the symmetry of $H$. For all $\psi_1, \psi_2\in \mathcal{D}(H)$ we have $H\psi_1, H\psi_2 \in L^2(\Rd,\C^4)$, which we denote by $g_1, g_2$ respectively. Let
$$f_1=g_1+ i \psi_1\quad \text{and} \quad f_2=g_2- i\psi_2$$
which belong to $L^2(\Rd,\C^4)$. By Theorem \ref{teorema}, $\psi_1$ and $\psi_2$ are the unique functions in $L^2(\Rd,\C^4)$ such that 
$$(H+ i)\psi_1=f_1, \quad (H- i)\psi_2=f_2.$$
Now write
\begin{eqnarray*}
\langle H\psi_1, \psi_2\rangle&=&\langle (H+ i)\psi_1, \psi_2\rangle - \langle i\psi_1, \psi_2\rangle\\
&=&\langle \psi_1, (H- i)\psi_2\rangle - \langle i\psi_1, \psi_2\rangle\\
&=&\langle \psi_1, H\psi_2\rangle.
\end{eqnarray*}
Note that in the second identity we have used (v) of Theorem \ref{teorema}.

To conclude the proof we need to characterize the domain of the extension, thus we have to see that for all $\psi\in \mathcal{D}(H)$,
$$\int_{\Rd}|\psi|^2\frac{1}{|x|}<+\infty.$$
However, since $\psi\in \mathcal{D}(H)$, $\psi,H\psi\in L^2(\Rd,\C^4)$, so that
$f=(H+i)\psi$ belongs to $L^2(\Rd,\C^4)$. By (ii) of Theorem \ref{teorema} we get the desired result. It is easy to check that $\mathcal{D}(H)\subset H^{1/2}(\Rd,\C^4)$ thanks to (iv) of Theorem \ref{teorema}.

\section{Further comments}\label{sec:further}

In this section we follow the arguments of \cite{DDEV} and give an alternative proof of Theorem \ref{final_ineq}. We denote by $X_+$ (resp. $X_-$) the positive (resp. negative) spectral space of $1+\sigma\cdot L$ and by $P_{\pm}=\frac{1}{2}\left(1\pm\frac{1+\sigma\cdot L}{|1+\sigma\cdot L|}\right)$ the corresponding projectors on $L^2(\Rd, \C^2)$. We write $\phi_{\pm}:= P_{\pm}\phi$. We define $\chi_{\pm}$ in the same way and we denote $\psi_{\pm}=(\phi_{\pm},\chi_{\pm})^{\perp}$.  Let $V(|x|)$ a radial positive function and we define the following radial functions
$$h^+(r)=\frac{1}{r^2}\int_0^rV(t)t^2dt\quad \text{and}\quad h^-(r)=r^2\int_r^{\infty}V(t)\frac{dt}{t^2}.$$

\begin{thm}\label{otraprueba}
Let $V(|x|)$ a radial positive function such that 
$\|h^+\|_{\infty}\leq \frac{1}{2}$ and $\|h^-\|_{\infty}\leq \frac{1}{2}$. Then, 
\begin{equation}\label{hip}
\int_{\Rd}V|\psi|^2\leq\int_{\Rd}V^{-1}|(\alpha\cdot\nabla+im\beta\pm\epsilon\I_4)\psi|^2+ \sqrt{m^2+\epsilon^2}\left(\int_{\Rd}|(h^--h^+)\psi|^2\right)^{1/2}\left(\int_{\Rd}|\psi|^2\right)^{1/2}.
\end{equation}
In particular, if $V=\frac{1}{|x|}$ then $h^+=h^-=\frac{1}{2}$ and we recover Theorem \ref{final_ineq}.
\end{thm}

\begin{proof}
As in Theorem \ref{final_ineq} here also we can assume that $\epsilon=1$ without loss of generality. We will write the proof just for the plus sign of $\I_4$. The corresponding inequality for the minus sign is obtained following the same argument. The desired inequality is obtained by estimating from below and from above the following expression
\begin{equation}\label{real}
2 \Re \int_{\Rd}(\alpha\cdot\nabla+im\beta+\I_4)\psi \cdot\overline{\left(\alpha\cdot\frac{x}{|x|}\right)(-h^+ \psi_++h^-\psi_-)}.
\end{equation}
Let us first estimate it from below. We write the real part as a sum of the integral and its conjugate and we develop it obtaining the following
\begin{align*}
&- \int_{\Rd}(\alpha\cdot\nabla+im\beta+\I_4)\psi_+ \cdot\overline{\left(\alpha\cdot\frac{x}{|x|}\right)h^+ \psi_+}-\int_{\Rd}\left(\alpha\cdot\frac{x}{|x|}\right)h^+ \psi_+ \cdot\overline{(\alpha\cdot\nabla+im\beta+\I_4)\psi_+}\\
&+\int_{\Rd}(\alpha\cdot\nabla+im\beta+\I_4)\psi_- \cdot\overline{\left(\alpha\cdot\frac{x}{|x|}\right)h^- \psi_-}+\int_{\Rd}\left(\alpha\cdot\frac{x}{|x|}\right)h^- \psi_- \cdot\overline{(\alpha\cdot\nabla+im\beta+\I_4)\psi_-}\\
&+\int_{\Rd}(\alpha\cdot\nabla+im\beta+\I_4)\psi_+ \cdot\overline{\left(\alpha\cdot\frac{x}{|x|}\right)h^- \psi_-}+\int_{\Rd}\left(\alpha\cdot\frac{x}{|x|}\right)h^- \psi_- \cdot\overline{(\alpha\cdot\nabla+im\beta+\I_4)\psi_+}\\
&-\int_{\Rd}(\alpha\cdot\nabla+im\beta+\I_4)\psi_- \cdot\overline{\left(\alpha\cdot\frac{x}{|x|}\right)h^+ \psi_+}-\int_{\Rd}\left(\alpha\cdot\frac{x}{|x|}\right)h^+ \psi_+ \cdot\overline{(\alpha\cdot\nabla+im\beta+\I_4)\psi_-}.
\end{align*}
We denote the first fourth terms of the last equation by $I$ and the last fourth by $II$. First we will study $II$. 
\begin{align*}
II=&\int_{\Rd}\alpha\cdot\nabla\psi_+ \cdot\overline{\left(\alpha\cdot\frac{x}{|x|}\right)h^- \psi_-}
+\int_{\Rd}\left(\alpha\cdot\frac{x}{|x|}\right)h^- \psi_- \cdot\overline{\alpha\cdot\nabla\psi_+}\\
&-\int_{\Rd}\alpha\cdot\nabla\psi_- \cdot\overline{\left(\alpha\cdot\frac{x}{|x|}\right)h^+ \psi_+}-\int_{\Rd}\left(\alpha\cdot\frac{x}{|x|}\right)h^+ \psi_+ \cdot\overline{\alpha\cdot\nabla\psi_-}\\
&+\int_{\Rd}(im\beta+\I_4)\psi_+ \cdot\overline{\left(\alpha\cdot\frac{x}{|x|}\right)h^- \psi_-}+\int_{\Rd}\left(\alpha\cdot\frac{x}{|x|}\right)h^- \psi_- \cdot\overline{(im\beta+\I_4)\psi_+}\\
&-\int_{\Rd}(im\beta+\I_4)\psi_- \cdot\overline{\left(\alpha\cdot\frac{x}{|x|}\right)h^+ \psi_+}-\int_{\Rd}\left(\alpha\cdot\frac{x}{|x|}\right)h^+ \psi_+ \cdot\overline{(im\beta+\I_4)\psi_-}.
\end{align*}
The first four terms vanish because $\alpha\cdot\nabla$ anti-commutes with the positive projector, hence we obtain the negative projection of another term. Analogously for the negative projector. The same thing happens with $\alpha\cdot\frac{x}{|x|}$ because we can write this in terms of $\alpha\cdot\nabla$ as follows
$$\alpha\cdot\frac{x}{|x|}=e^{|x|}\alpha\cdot\nabla e^{-|x|}-\alpha\cdot\nabla.$$
Hence, we will have the inner product of a negative projection of a term with a positive  projection of another term, which is zero because of their orthogonality. Therefore,
\begin{eqnarray*}
II&=&\int_{\Rd}im\beta\psi_+ \cdot\overline{\left(\alpha\cdot\frac{x}{|x|}\right)h^- \psi_-}+\int_{\Rd}\left(\alpha\cdot\frac{x}{|x|}\right)h^- \psi_- \cdot\overline{im\beta\psi_+}\\
&-&\int_{\Rd}im\beta\psi_- \cdot\overline{\left(\alpha\cdot\frac{x}{|x|}\right)h^+ \psi_+}-\int_{\Rd}\left(\alpha\cdot\frac{x}{|x|}\right)h^+ \psi_+ \cdot\overline{im\beta\psi_-}\\
&+&\int_{\Rd}\psi_+ \cdot\overline{\left(\alpha\cdot\frac{x}{|x|}\right)h^- \psi_-}+\int_{\Rd}\left(\alpha\cdot\frac{x}{|x|}\right)h^- \psi_- \cdot\overline{\psi_+}\\
&-&\int_{\Rd}\psi_- \cdot\overline{\left(\alpha\cdot\frac{x}{|x|}\right)h^+ \psi_+}-\int_{\Rd}\left(\alpha\cdot\frac{x}{|x|}\right)h^+ \psi_+ \cdot\overline{\psi_-}.
\end{eqnarray*}
Now taking into account that $\alpha\cdot\frac{x}{|x|}$ is symmetric, that $\alpha_k\beta=-\beta\alpha_k$ and by using the Cauchy-Schwarz inequality we get
\begin{eqnarray*}
II&=&\int_{\Rd}\left(\alpha\cdot\frac{x}{|x|}\right)im\beta\psi_+ \cdot\overline{h^- \psi_-}-\int_{\Rd}im\beta\left(\alpha\cdot\frac{x}{|x|}\right) \psi_- \cdot\overline{h^-\psi_+}\\
&-&\int_{\Rd}\left(\alpha\cdot\frac{x}{|x|}\right)im\beta\psi_- \cdot\overline{h^+ \psi_+}+\int_{\Rd}im\beta\left(\alpha\cdot\frac{x}{|x|}\right) \psi_+ \cdot\overline{h^+\psi_-}\\
&+&\int_{\Rd}\psi_+ \cdot\overline{\left(\alpha\cdot\frac{x}{|x|}\right)h^- \psi_-}+\int_{\Rd}\psi_- \cdot\overline{\left(\alpha\cdot\frac{x}{|x|}\right)h^- \psi_+}\\
&-&\int_{\Rd}\psi_- \cdot\overline{\left(\alpha\cdot\frac{x}{|x|}\right)h^+ \psi_+}-\int_{\Rd} \psi_+ \cdot\overline{\left(\alpha\cdot\frac{x}{|x|}\right)h^+\psi_-}\\
&=&2\Re \int_{\Rd}(im\beta+\I_4)\psi_+\cdot\overline{\left(\alpha\cdot\frac{x}{|x|}\right)(h^--h^+) \psi_-}\\
&\leq&\sqrt{m^2+\epsilon^2}\left(\int_{\Rd}|(h^--h^+)\psi|^2\right)^{1/2}\left(\int_{\Rd}|\psi|^2\right)^{1/2}.
\end{eqnarray*}

Let us estimate $I$ from below. Recall that $I$ was the following expression
\begin{align*}
&-  \int_{\Rd}(\alpha\cdot\nabla+im\beta+\I_4)\psi_+ \cdot\overline{\left(\alpha\cdot\frac{x}{|x|}\right)h^+ \psi_+}-\int_{\Rd}\left(\alpha\cdot\frac{x}{|x|}\right)h^+ \psi_+ \cdot\overline{(\alpha\cdot\nabla+im\beta+\I_4)\psi_+}\\
&+\int_{\Rd}(\alpha\cdot\nabla+im\beta+\I_4)\psi_- \cdot\overline{\left(\alpha\cdot\frac{x}{|x|}\right)h^- \psi_-}+\int_{\Rd}\left(\alpha\cdot\frac{x}{|x|}\right)h^- \psi_- \cdot\overline{(\alpha\cdot\nabla+im\beta+\I_4)\psi_-}.
\end{align*}
Observe that by using the previous argument the integrals related to $im\beta+\I_4$ vanish. This is because we commute the projectors $P_{\pm}$ with $\alpha\cdot\frac{x}{|x|}$ and we obtain the inner product of two orthogonal terms. Note that $\beta$ does not change the spectral space. In consequence,
\begin{eqnarray*}
I&=&-  \int_{\Rd}\alpha\cdot\nabla\psi_+ \cdot\overline{\left(\alpha\cdot\frac{x}{|x|}\right)h^+ \psi_+}-\int_{\Rd}\left(\alpha\cdot\frac{x}{|x|}\right)h^+ \psi_+ \cdot\overline{\alpha\cdot\nabla\psi_+}\\
&+&\int_{\Rd}\alpha\cdot\nabla\psi_- \cdot\overline{\left(\alpha\cdot\frac{x}{|x|}\right)h^- \psi_-}+\int_{\Rd}\left(\alpha\cdot\frac{x}{|x|}\right)h^- \psi_- \cdot\overline{\alpha\cdot\nabla\psi_-}.
\end{eqnarray*}
We write all the integrals in terms of $\sigma$ and  use
$$[\sigma\cdot\nabla,\sigma\cdot\frac{x}{|x|} h^+]=2(1+\sigma\cdot L)\frac{h^+}{|x|}+\frac{h^+}{|x|}+|x|\left(\frac{h^+}{|x|}\right)'$$
where the derivative is with respect to the radius $|x|$, see \cite{DELV} for the details of this identity. As a consequence, we get
\begin{align*}
I=&\int_{\Rd}[\sigma\cdot\nabla,\sigma\cdot\frac{x}{|x|} h^+] \phi_+\cdot \overline{\phi_+}+\int_{\Rd}[\sigma\cdot\nabla,\sigma\cdot\frac{x}{|x|} h^+] \chi_+\cdot \overline{\chi_+}\\
&-\int_{\Rd} [\sigma\cdot\nabla,\sigma\cdot\frac{x}{|x|} h^-] \phi_-\cdot \overline{\phi_-}-\int_{\Rd} [\sigma\cdot\nabla,\sigma\cdot\frac{x}{|x|} h^-] \chi_-\cdot \overline{\chi_-}\\
=&\int_{\Rd}[2(1+\sigma\cdot L)\frac{h^+}{|x|}+(h^+)']\phi_+\cdot \overline{\phi_+}+\int_{\Rd}[2(1+\sigma\cdot L)\frac{h^+}{|x|}+(h^+)']\chi_+\cdot \overline{\chi_+}\\
&-\int_{\Rd}[2(1+\sigma\cdot L)\frac{h^-}{|x|}+(h^-)']\phi_-\cdot \overline{\phi_-}-\int_{\Rd}[2(1+\sigma\cdot L)\frac{h^-}{|x|}+(h^-)']\chi_-\cdot \overline{\chi_-}\\
\geq& \int_{\Rd}\left(\frac{2}{|x|}h^++(h^+)'\right)\psi_+\cdot\overline{\psi_+}+\int_{\Rd}\left(\frac{2}{|x|}h^--(h^-)'\right)\psi_-\cdot\overline{\psi_-}.
\end{align*}
Now,  by definition of $h^+(r)$ and  $h^-(r)$ it is immediate that 
$$\frac{2}{|x|}h^++(h^+)'=V(r)=\frac{2}{|x|}h^--(h^-)'.$$
Therefore,
\begin{equation*}
I\geq \int_{\Rd}V|\psi_+|^2+ \int_{\Rd}V|\psi_-|^2= \int_{\Rd}V|\psi|^2.
\end{equation*}
Let us estimate (\ref{real}) from above. We multiply and divide by $V(|x|)$ inside the integral. We bound the real part by the modulus and apply the Cauchy-Schwarz inequality. Hence, we obtain 
\begin{eqnarray*}
&&2 \Re \int_{\Rd}(\alpha\cdot\nabla+im\beta+\I_4)\psi \cdot\overline{\left(\alpha\cdot\frac{x}{|x|}\right)(-h^+ \psi_++h^-\psi_-)}\\
&\quad& \leq2\left( \int_{\Rd}V^{-1}|(\alpha\cdot\nabla+im\beta+\I_4)\psi|^2\right)^{1/2}\left( \int_{\Rd}V|-h^+\psi_++h^-\psi_-|^2\right)^{1/2}\\
&\quad& \leq2\max\{\|h^+\|_{\infty},\|h^-\|_{\infty}\}\left( \int_{\Rd}V^{-1}|(\alpha\cdot\nabla+im\beta+\I_4)\psi|^2\right)^{1/2}\left( \int_{\Rd}V|\psi|^2\right)^{1/2}.
\end{eqnarray*}
Now taking into account upper and lower bounds of the real part we have
\begin{equation*}
\int_{\Rd}V(|x|)|\psi|^2\leq4\max\{\|h^+\|_{\infty}^2,\|h^-\|_{\infty}^2\}\int_{\Rd}V^{-1}(|x|)|(\alpha\cdot\nabla+im\beta+\I_4)\psi|^2.
\end{equation*}
Since $\|h^+\|_{\infty}\leq1/2$ and $\|h^-\|_{\infty}\leq1/2$ we conclude the proof.
\end{proof}

\begin{remark}
If $V=\frac{1}{|x|}$ then for $\psi_0^{\epsilon,m}$ where $\phi_0^{\epsilon,m}=C r^{-1}e^{-\sqrt{\epsilon^2+m^2} r}$, $C\in\C^2$ and $\chi_0^{\epsilon,m}=\frac{\pm\epsilon+im}{\sqrt{\epsilon^2+m^2}}\left(\sigma\cdot\frac{x}{|x|}\right)\phi_0^{\epsilon,m}$ all the inequalities in the above argument become an equalities.
\end{remark}

\section*{Acknowledgements}
The authors would like to thank M. J. Esteban and N. Visciglia for many enlightening conversations.


\begin{thebibliography}{99}
\bibitem{A1}
{\sc Arai, M.} On essential self-adjointness of Dirac operators. \textit{RIMS Kokyuroku, Kyoto Univ.} \textbf{242} (1975), 10--21.
\bibitem{A2}
{\sc Arai, M.} On essential self-adjointness, distinguished self-adjoint extension and essential spectrum of Dirac operators with matrix-valued potentials. \textit{Publ. RIMS, Kyoto Univ.} \textbf{19} (1983), 33--57.
\bibitem{A3}
{\sc Arai, M., Yamada, O.} Essential self-adjointness and invariance of the essential spectrum for Dirac operators. \textit{Publ. RIMS, Kyoto Univ.} \textbf{18} (1982), 973--985.
\bibitem{NA} 
{\sc Arrizabalaga, N.} Distinguished self-adjoint extensions of Dirac operators via Hardy-Dirac inequalities. \textit{J. Math. Phys.} \textbf{52} 092301 (2011).
\bibitem{BDF}
{\sc Boussaid, N., D'Ancona, P., Fanelli, F.}, Virial identity and weak dispersion for the magnetic dirac equation. \textit{Journ. Math. Pures Appl.} \textbf{95} (2011), 137--150.
\bibitem{DELV}
{\sc Dolbeault, J., Esteban, M.J., Loss, M., Vega, L.} An analytical proof of Hardy-like inequalities related to the Dirac operator. \textit{J. Funct. Anal.} \textbf{216} (2004), 1--21.
\bibitem{DDEV}
{\sc Dolbeault, J., Duoandikoetxea, J., Esteban, M.J., Vega, L.} Hardy-type estimates for Dirac operators. \textit{Ann. Scient. $\acute E$c. Norm. Sup.} \textbf{40} (2007), 885--900.
\bibitem{DEL} 
{\sc Dolbeault, J., Esteban, M.J., Loss, M.} Relativistic hydrogenic atoms in strong magnetic fields. \textit{Ann. Henri Poincar\'e} \textbf{8} (2007), 749--779.
\bibitem{DV}
 {\sc Duoandikoetxea, J., Vega, L.} Some weighted Gagliardo-Nirenberg inequalities and applications. \textit{Proc. Amer. Math. Soc.} \textbf{135} (2007), 2795--2802.
\bibitem{EL}
{\sc Esteban, M.J., Loss, M.} Self-adjointness for Dirac operators via Hardy-Dirac inequalities. \textit{J. Math. Phys.} \textbf{48}(11) (2007), 112107.
\bibitem{Ka}
{\sc Kato, T.}, \textit{Perturbation Theory of Linear Operators}, 2nd edition, Springer Verlag, Berlin, Heidelberg, New York 1980.
\bibitem{KW} 
{\sc Klaus, M., W\"ust, R.} Characterization and uniqueness of distinguished self-adjoint extensions of Dirac operators. \textit{Comm. Math. Phys.} \textbf{64} (1979), 171--176.
\bibitem{N}
{\sc Nenciu, G.} Self-adjointness and invariance of the essential spectrum for Dirac operators defined as quadratic forms. \textit{Comm. Math. Phys.} \textbf{48} (1976), 235--247.
\bibitem{S1}
{\sc Schmincke,  U.-W.},  Essential self-adjointness of Dirac operators with strongly singular potential. \textit{Math. Z.} \textbf{126} (1972), 71--81.
\bibitem{S2} 
{\sc Schmincke, U.-W.}, Distinguished self-adjoint extensions of Dirac operators. \textit{Math. Z.} \textbf{129} (1972) 335--349.
\bibitem{T}
{\sc Thaller, B.}, \textit{The Dirac Equation}, Springer-Verlag 1992.
\bibitem{V} 
{\sc Vogelsang, V.}, Remark on essential self-adjointness of Dirac operators with Coulomb potentials. \textit{Math. Z.} \textbf{196} (1987), 517--521.
\bibitem{W} 
{\sc Weidmann, J.}, Oszillationsmethoden f\"ur Systeme gew\"ohnlicher Differentialgleichungen. \textit{Math. Z.} \textbf{119} (1971), 349--373.
\bibitem{W1} 
{\sc W\"ust, R.}, A convergence theorem for self-adjoint operators applicable to Dirac operators with cut-off potentials. \textit{Math. Z.} \textbf{131} (1973), 339--349.

\bibitem{W2} 
{\sc W\"ust, R.}, Distinguished self-adjoint extensions of Dirac operators constructed by means of cut-off potentials. \textit{Math. Z.} \textbf{141} (1975), 93--98.

\bibitem{W3} 
{\sc W\"ust, R.}, Dirac operators with strongly singular potentials. \textit{Math. Z.} \textbf{152} (1977),  259--271.
\end{thebibliography}
\end{document}